\documentclass[a4paper]{article}
\usepackage{graphicx}
\usepackage{amssymb}
\usepackage{amsmath}
\usepackage{amsthm}
\usepackage{amscd}
\usepackage{color}
\usepackage{colordvi}
\usepackage[all,2cell]{xy}
\usepackage{tikz}
\usepackage{picture}
\usepackage{multirow}
\usepackage{makecell}
\usepackage{mathrsfs}
\usepackage{url}

\UseAllTwocells \SilentMatrices
\newtheorem{thm}{Theorem}[section]

\newtheorem{cor}[thm]{Corollary}

\newtheorem{lem}[thm]{Lemma}
\theoremstyle{definition}
\newtheorem{exm}[thm]{Example}
\theoremstyle{plain}

\newtheorem{prop}[thm]{Proposition}
\theoremstyle{definition}
\newtheorem{defn}[thm]{Definition}
\theoremstyle{remark}
\newtheorem{rem}[thm]{\bf Remark}
\numberwithin{equation}{section}

\newcommand{\Tor}{\operatorname{Tor}}

 \DeclareMathOperator{\Ext}{Ext}
 
 \DeclareMathOperator{\gldim}{gl.dim}
 
 \DeclareMathOperator{\Hom}{Hom}

 \DeclareMathOperator{\injdim}{inj.dim}

 \DeclareMathOperator{\projdim}{proj.dim}

 \DeclareMathOperator{\RHom}{\mathbf{RHom}}

 \newcommand{\otimesL}{\otimes^{\bf L}}

\newcommand{\defCategory}[2]{
  \newcommand{#1}{#2\defvariable}}

\newcommand{\defvariable}[2][]{
\if\relax\detokenize{#1}\relax  
\if\relax\detokenize{#2}\relax
    \else  ({#2})  \fi
    \else  ^{{\rm #1}}({#2})  \fi}

\defCategory{\C}{\mathscr{C}}
\defCategory{\K}{\mathscr{K}}
\defCategory{\D}{\mathscr{D}}

 \def\Kb#1{\K[b]{#1}}

 \def\Db#1{\D[b]{#1}}

\def\Dsg#1{\mathscr{D}_{\sf sg}(#1)}

\foreach \x in {A,...,Z}{%
\expandafter\xdef\csname cal\x\endcsname{\noexpand\ensuremath{\noexpand\mathcal{\x}}}
\expandafter\xdef\csname scr\x\endcsname{\noexpand\ensuremath{\noexpand\mathscr{\x}}}
\expandafter\xdef\csname bb\x\endcsname{\noexpand\ensuremath{\noexpand\mathbb{\x}}}}

\def\modcat#1{{#1}\mbox{-}{\sf mod}}
\def\Modcat#1{{#1}\mbox{-}{\sf Mod}}
\def\pmodcat#1{{#1}\mbox{-}{\sf proj}}
\def\Pmodcat#1{{#1}\mbox{-}{\sf Proj}}
\def\GPmodcat#1{{#1}\mbox{-}{\sf Gproj}}

\def\imodcat#1{{#1}\mbox{-}{\sf inj}}

\newcommand{\lra}{\longrightarrow}

\newcommand{\ra}{\rightarrow}

\newcommand{\cpx}[1]{#1^{\bullet}}
\newcommand{\opp}{^{\rm op}}

\title{Singular equivalences and Auslander-Reiten conjecture}
\author {Yiping Chen, Wei Hu, Yongyun Qin, Ren Wang}
\date{\today}

\begin{document}

\maketitle


\begin{abstract}
Auslander-Reiten conjecture, which says that an Artin algebra does not have any non-projective generator with vanishing self-extensions in all positive degrees, is shown to be invariant under certain singular equivalences induced by adjoint pairs, which occur often in matrix algebras, recollements and change of rings. Accordingly, several reduction methods are established to study this conjecture.
\end{abstract}

\section{Introduction}

In the representation theory of algebras, the celebrated  Nakayama conjecture ({\bf NC} for short) states that
an Artin algebra $A$ is self-injective provided that all terms in a minimal injective
resolution of $A$ are projective. In \cite{Muller1968aa}, M\"{u}ller restated the Nakayama conjecture and proved that the Nakayama conjecture holds for all algebras if and only if for each  algebra $A$, a generator-cogenerator $M$ of $A$ with $\Ext_{A}^{i}(M, M)$ vanishing for all $i>0$ is necessarily projective.  In \cite{Auslander1975ae}, Auslander and Reiten proposed a conjecture which is an analogue of M\"{u}ller' theorem.

\medskip
  {\bf Auslander-Reiten conjecture} ({\bf ARC}):  {\em A generator (may not be a cogenerator) $M$ of $A$ with  $\Ext_{A}^{i}(M, M)=0$ for all $i>0$ must be projective, or equivalently, a finitely generated $A$-module $X$ satisfying $\Ext_A^i(X, X\oplus A)=0$ for all $i>0$ is necessarily projective.}

\medskip
 ARC is closely connected with NC and several famous conjectures in the representation theory of Artin algebras. For example, let $A$ be an Artin algebra.

\medskip
{\bf Finitistic dimension conjecture(FDC)}: {\em The finitistic dimension of $A$ (the supremum of the projective dimensions of finitely generated $A$-modules with finite projective dimension) is finite. }

\smallskip
{\bf Strong Nakayama conjecture(SNC)}: {\em  For each nonzero finitely generated $A$-module $M$, there is an integer $n\geqslant 0$ such that $\Ext^n_A(M, A)\neq 0$ } \cite{CF1990}.

\smallskip
{\bf Generalized Nakayama conjecture(GNC)}: {\em  Every indecomposable injective $A$-module occurs as a direct summand of a term in a minimal injective resolution of ${}_AA$} \cite{Auslander1975ae}.

\medskip
 All these conjectures are widely open. The relationship among them is stated as follows. For more details, we refer the readers to \cite{Auslander1975ae}, \cite[Theorem 3.4.3]{Yamagata1996}.
 \begin{itemize}
 	\setlength{\itemsep}{0em}
 \item For an individual algebra,  FDC$\Rightarrow$ SNC$\Rightarrow$ GNC $\Rightarrow$ NC.

 \item
GNC holds for all Artin algebras if and only if so does ARC for all Artin algebras.
 \end{itemize}
 Thus, FDC holds for all algebras implies that ARC holds for all algebras.
 However, the implication is unknown for individual algebras. For example, the
 finitistic dimension conjecture is known true for
 self-injective algebras, but ARC is still open for these algebras \cite[Theorem 3.4]{Hoshino1984}.

ARC is known to be true for several special classes of algebras. For example, algebras of finite representation type, torsionless-finite algebras, symmetric biserial algebras, algebras with radical square zero, local algebras with radical cube zero \cite{Auslander1975ae, Xu2013a, Xu2015}.  ARC also holds for an algebra satisfying the ({\bf{Fg}}) condition \cite{EHTSS2004}. This kinds of algebras include group algebras of finite groups \cite{Carlson1983}, finite group schemes \cite{Friedlander1997b}, commutative complete intersections \cite{Avramov2000a}, quantum complete intersections where $q_{i,j}$ are roots of unity \cite{Oppermann2010}, and so on.

\medskip
This article is devoted to studying ARC in the context of singularity categories of algebras. The singularity category  $\Dsg{A}$ of an algebra $A$  is defined as the Verdier quotient
$$\Dsg{A}=\Db{\modcat{A}}/\Kb{\pmodcat{A}}$$
of the bounded derived category $\Db{\modcat{A}}$ by the subcategory of bounded complexes of projective modules \cite{Buchweitz1986}. Two algebras are called {\em singularly equivalent} if their singularity categories are equivalent as triangulated categories.

It is well-known that there is a full embedding
$$\underline{{}^{\perp}A}\hookrightarrow \Dsg{A}$$
sending a module to the corresponding stalk complex concentrated in degree zero,
where $${}^{\perp}A:=\{X\in\modcat{A}|\Ext_A^i(X, A)=0\mbox{ for all }i>0\}$$
 and $\underline{{}^{\perp}A}$ is the additive quotient category of ${}^{\perp}A$ modulo projective modules. The embedding induces an isomorphism
$$\Ext_A^i(X, Y)\cong \Hom_{\Dsg{A}}(X, Y[i])$$
for each $X, Y\in {}^{\perp}A$ and for each $i>0$.

ARC holds for $A$ precisely means that  $\underline{{}^{\perp}A}$ has no nonzero objects $X$ with $\Ext_A^i(X, X)=0$ for all positive $i$. Recall that an object $T$
in a triangulated category $\mathcal{T}$ is {\em presilting} if
$\Hom_{\mathcal{T}}(T, T[i])=0$ for all $i>0$.
As a consequence,
if $\Dsg{A}$ has no nonzero presilting objects, then ARC holds for $A$.
Therefore, it is natural to conjecture:
{\it The singularity category of any algebra contains no nonzero presilting objects}.  We call this singular presilting conjecture (SPC). Observe that  SPC implies  ARC, and
the converse is also true if $A$ is a Gorenstein algebra --- the above embedding is an equivalence in this case.
Hence SPC might be
used as a tool to study ARC.  Obviously, SPC is invariant under singular equivalences. This implies that ARC is preserved under singular equivalences
between Gorenstein algebras. One can ask a more general question.

\medskip
\noindent
{\em{\bf Question:}  Do singular equivalences preserve ARC?}

\medskip
If two algebras are derived equivalent (they are certainly also singularly equivalent), then the answer to the above question is yes. This was proved by Wei in \cite{Wei2012aa} (see \cite{Pan2013ac,Diveris2012aa} for derived invariance of a generalized version of ARC).  In this paper, we shall consider singular equivalences induced by adjoint pairs and show that many singular equivalences do preserve ARC.

\medskip
 Our first result is the following, which is listed in Theorem~\ref{ARE}.
\begin{thm}\label{thm-adjoint}
Let $A$, $B$ be two algebras. Suppose that there are triangle functors
\[\xymatrix{
	\D{\Modcat{B}}\ar[rr]^{G}\ar@/_2pc/[rr]_{K}\ar@/^2pc/[rr]^{L}&&\D{\Modcat{A}}\ar@/^1pc/[ll]^{H}\ar@/_1pc/[ll]_{F}&& }\]
between the unbounded derived categories of $A$ and $B$ such that $(L, F), (F, G), (G, H)$ and $(H, K)$ are all adjoint pairs. Assume that $G$ induces a singular equivalence, and that $H$ preserves bounded complexes of projective modules. Then the ARC holds for $A$ if and only if it holds for $B$.
\end{thm}

Applying Theorem \ref{thm-adjoint} to
ladders (see Section 3),
we get the following theorem (see Theorem~\ref{lad}).
\begin{thm}\label{lad-thm}
Let $A$, $B$ and $C$ be three algebras, and there is a ladder of height $3$
\[\xymatrix{
	\D{\Modcat{B}}\ar@/^1pc/[rr]|{i_*}\ar@/_1pc/[rr]|{i_?}&&
\D{\Modcat{A}} \ar[ll]|{i^!}\ar@/_2pc/[ll]|{i^*}\ar@/^2pc/[ll]
\ar@/^1pc/[rr]|{j^*}\ar@/_1pc/[rr]|{j^?}&&
\D{\Modcat{C}}\ar[ll]|{j_*}\ar@/_2pc/[ll]|{j_!}\ar@/^2pc/[ll]
}.\]

\medskip

\noindent
Then ARC holds for $A$ implies that it holds for $B$.
Moreover, if the ladder can be completed to a ladder of height $4$,
then the following statements hold.
\begin{enumerate}
	\item[{\rm (1).}] If the ARC holds on $A$, then it holds on $B$ and $C$;	
	\item[{\rm (2).}] If $C$ (resp. $B$) has finite global dimension, then ARC holds on $A$ if and only if it holds on $B$ (resp. $C$).
\end{enumerate}
\end{thm}

Theorem~\ref{thm-adjoint} can also be applied to singular equivalences
induced by tensor functors. The following result is listed in Theorem~\ref{atar}.
\begin{thm}\label{atar1}
Suppose that $\cpx{X}$ is a $B$-$A$-bimodules complex which is perfect over $B$ and $A$,
and assume $\cpx{Y}:= \RHom_B(\cpx{X},B)$ is a perfect complex of $A$-modules. If $\cpx{X}\otimesL_A-$ induces singular equivalences between $A$ and $B$, then ARC holds for $B$ implies that it holds for $A$. If moreover
$\RHom_A(\cpx{Y}, A)\in \Kb{\pmodcat{B}}$, then ARC holds for $A$ if and only if it holds for $B$.
\end{thm}

Our result can be applied to give several reduction methods on ARC. Let $A$ be a lower triangular matrix, and $e$ be an idempotent. Under certain conditions, $A$ satisfies ARC if and only if so does $eAe$ (see Corollary \ref{tri-alg}). As an immediate application, ARC is preserved under one-point (co-)extensions (see Example \ref{one-point extensions}). Thus, for quiver algebras, ARC is invariant when we remove a sink or a source. Let $AeA$ be a hereditary ideal of $A$ such that $eA$ has finite injective dimension as a right $A$-module. We also show that ARC holds for $A$ if and only if it holds for $A/AeA$ (see Example \ref{hereditary ideals}).

\medskip
This paper is organized as follows. We collect necessary facts in Section 2, and investigate ARC under singular equivalences induced by adjoint pairs and recollements in Section 3 and 4. Some examples are given in the final section.

\section{Preliminaries}
In this section we fix our notation and recall some basic facts for later proofs.

As we mentioned at the beginning of the previous section, all algebras are finite dimensional algebras over a fixed field $k$.  Let $A$ be such an algebra. The opposite algebra of $A$ is denoted by $A\opp$ and the enveloping algebra $A\otimes_k A\opp$ is denoted by $A^e$. We identify $A$-$A$-bimodules with left $A^e$-modules.
 We denote by $\Modcat{A}$ the category of left $A$-modules, by $\modcat{A}$ its subcategory consisting of finitely generated left $A$-modules, by $\Pmodcat{A}$ its subcategory consisting of projective left $A$-modules, and by $\pmodcat{A}$ its subcategory consisting of finitely generated projective left $A$-modules.

Let $\mathscr{X}$ be a full subcategory of $\Modcat{A}$. If $\mathscr{X}$ is not contained in $\modcat{A}$, then we define
$${}^{\perp}\mathscr{X}:=\{Y\in\Modcat{A}\mid \Ext_{A}^{i}(Y, X)=0, \forall i>0, X\in\mathscr{X}\},$$
otherwise, we write
$${}^{\perp}\mathscr{X}:=\{Y\in\modcat{A}\mid \Ext_{A}^{i}(Y, X)=0, \forall i>0, X\in\mathscr{X}\}.$$
The corresponding additive quotient modulo projectives is denoted by $\underline{{}^{\perp}\mathscr{X}}$.  An $A$-module $M$ contained in ${}^{\perp}M$ is called {\em self-orthogonal.}

\medskip
The following result is well-known.

\begin{lem}\label{fg}
Let $A$ be an algebra. Then \[{}^{\perp}(\Pmodcat{A})\cap \modcat{A}={}^{\perp}A.\]
\end{lem}
\begin{proof}
It is clear that the left hand is contained in the right hand. For any $X\in {}^{\perp}A\subseteq \modcat{A}$, one has
\[{\rm Ext}_A^i(X,\coprod\limits_{j\in I} A)\simeq \coprod\limits_{j\in I}{\rm Ext}_A^i(X,A)\]
for any index set $I$. Thus, $X\in {}^{\perp}(\Pmodcat{A}).$
\end{proof}

For an algebra $A$, the unbounded derived category of $\Modcat{A}$ is denoted by $\D{\Modcat{A}}$. As usual, we write $\Db{\Modcat{A}}$ (resp. $\Db{\modcat{A}}$) for the full subcategory consisting of bounded complexes
of left $A$-modules (resp. finitely generated left $A$-modules), and write $\Kb{\Pmodcat{A}}$ (resp. $\Kb{\pmodcat{A}}$) for the full subcategory consisting of bounded complexes of projective modules (resp. finitely generated projective modules).
 From \cite{Krause2005ab}, the ``big" singularity category
of $A$ is defined to be the following Verdier quotient
$$\Dsg{\Modcat{A}}=\Db{\Modcat{A}}/\Kb{\Pmodcat{A}}.$$
It is well-known that the above mentioned unbounded and bounded derived categories, and the singularity category are all triangulated categories. We refer to Happel's book \cite{Happel1988aa} for basic results on triangulated categories. For an object $X$ in a triangulated category, we write $X[n]$ for the object obtained from $X$ by applying the shift functor $n$ times.
The object $X$ is called {\em presilting} if $\Hom(X, X[n])=0$ for all $n>0$.

\medskip
Note that there is a full embedding
$$\underline{{}^{\perp}(\Pmodcat{A})}\hookrightarrow \Dsg{\Modcat{A}}$$
which induces a natural isomorphism
$$\Ext_{A}^{i}(X, Y)\cong \Hom_{\Dsg{\Modcat{A}}}(X, Y[i])$$
for each $i>0$. Thus, every self-orthogonal object in $\underline{{}^{\perp}(\Pmodcat{A})}$ is a presilting object in $\Dsg{\Modcat{A}}$.

A complex $\cpx{X}$ of (finitely generated) $A$-modules is a sequence of (finitely generated) $A$-module homomorphisms
\[\cpx{X}: \cdots\rightarrow X^i\overset{d^i}{\rightarrow} X^{i+1}\overset{d^{i+1}}{\rightarrow} X^{i+2}\rightarrow \cdots\]
such that $d^{i+1} d^i=0$ for all $i\in \mathbb{Z}$. The $i$-th homology of $\cpx{X}$ is denoted by $H^i(\cpx{X})$.
For a right $A$-module $M$ and a left $A$-module $N$, denote by $M\otimes_A \cpx{X}$ and $\Hom_A(\cpx{X},N)$ the complexes
\[\cdots\rightarrow M\otimes_A X^i\overset{1\otimes d^i}{\longrightarrow} M\otimes_A X^{i+1}\rightarrow \cdots\]
and
\[\cdots\rightarrow\Hom_A(X^{i+1},N)\overset{\Hom_A(d^i,N)}{\longrightarrow} \Hom_A(X^i,N)\rightarrow \cdots\]
respectively. Note that the $i$-th term of $\Hom_A(\cpx{X},N)$ is $\Hom_A(X^{-i},N)$ for all $i\in \mathbb{Z}$.

\medskip

The following  will be useful in our later discussion.

\begin{lem}
[{\cite[Theorem 4.1 and 5.1]{Neeman1996ab} and  \cite[Lemma 2.7]{Angeleri-Hugel2017aa}}]	\label{re1}
Let $A$ and $B$ be two  algebras, and $F :\D{\Modcat{A}}\rightarrow \D{\Modcat{B}}$ be a triangle functor with a right adjoint $G$.
Consider the following conditions

$(1)$ $F$ preserves $\Kb{\sf proj}$;

$(2)$ $G$ preserves coproducts;

$(3)$ $G$ admits a right adjoint;

$(4)$ $G$ preserves $\Db{\sf mod}$;

$(5)$ $G$ preserves $\Db{\sf Mod}$.

Then we have $(1)\Leftrightarrow (2)\Leftrightarrow (3)\Leftrightarrow (4)\Rightarrow(5)$.
\end{lem}
%
%

\section{Singular equivalences induced by adjoint tuples}
In this section, we consider ARC  and singular equivalences induced by adjoint pairs. First, let us recall from \cite{Hu2017aa} the definition of non-negative functors.

\begin{defn}[{\cite[Definition 4.1]{Hu2017aa}}]\label{def-non}
Let $A$ and $B$ be two algebras.   A triangle functor
$$G :\Db{\Modcat{B}}
\rightarrow \Db{\Modcat{A}}$$
 is called \emph{non-negative} if $G$ satisfies the following two conditions:
 \setlength{\itemsep}{0pt}
 \begin{enumerate}
 	\item[(1)] $G(X)$ is isomorphic to a complex with zero homology in all negative degrees for all $X\in \Modcat{B}$;
 	\item[(2)] $G(P)$ is isomorphic to a complex in $\Kb{\Pmodcat{A}}$ with zero terms in all negative degrees for all $P\in \Pmodcat{B}$.
 \end{enumerate}
 \end{defn}

\medskip

The following proposition taken from \cite{Hu2017aa} on non-negative functors will be crucial for our discussion.

\begin{prop}[{\cite[Proposition 4.8 and 5.2]{Hu2017aa}}]\label{cd}
	Let $G:\Db{\Modcat{B}}\ra\Db{\Modcat{A}}$
	be a non-negative triangle functor admitting a right adjoint $H$ which preserving $\Kb{\rm Proj}$. Then there is a  commutative diagram
	\begin{equation}\label{d}
\xymatrix@C=1.5cm{
	\underline{{}^{\perp}(\Pmodcat{B})} \ar[r]^{\overline{G}} \ar @{^{(}->}[d]^{\iota_B}
	& \underline{{}^{\perp}(\Pmodcat{A})} \ar @{^{(}->}[d]^{\iota_A}  \\
	\Dsg{\Modcat{B}}  \ar[r]^{\tilde{G}}   &\Dsg{\Modcat{A}},}
\end{equation}
	where $\iota_B$ and $\iota_A$ are natural embeddings.
\end{prop}

In certain cases, the above commutative diagram indicates the relationship between the validity of ARC for  $A$ and $B$, as is shown in the the following lemma.

\begin{lem}\label{cdar}
	Assume that we have the above commutative diagram {\rm (\ref{d})}. If $\overline{G}$ preserves finitely generated modules, and $\tilde{G}$ is fully faithful, then ARC  holds for $A$ implies that  it holds for $B$.
\end{lem}
\begin{proof}
    Let $X$ be a self-orthogonal object in $\underline{{}^{\perp}B}$. Then $X$ is finitely generated and contained in $\underline{{}^{\perp}(\Pmodcat{B})}$ by Lemma \ref{fg}.  Since the embedding $\iota_{B}$  takes self-orthogonal objects to presilting objects, the object $\iota_{B}(X)$ and thus $\tilde{G}\iota_{B}(X)$ is presilting. By the commutative diagram (\ref{d}), one has $\iota_{A}\bar{G}(X)=\tilde{G}\iota_{B}(X)$. It follows that $\bar{G}(X)$ is self-orthogonal. By our assumption $\bar{G}(X)$ is still finitely generated and thus belongs to $\underline{{}^{\perp}A}$ by Lemma \ref{fg}.  By assumption ARC holds for $A$. This implies that $\bar{G}(X)$ is isomorphic to the zero object in  $\underline{{}^{\perp}A}$. Since $\tilde{G}$ is fully faithful, by the commutative diagram (\ref{d}), one can see that $\bar{G}$ is fully faithful. Hence $X$ must be isomorphic to the zero object in $\underline{{}^{\perp}B}$.

    Altogether, we have shown that $\underline{{}^{\perp}B}$ does not contain any nonzero self-orthogonal objects, that is, ARC holds for $B$.
\end{proof}

Lemma \ref{cdar} will serve as our main idea to study the relationship between singular equivalences and ARC. However, at this stage, we don't know how to get non-negative
functors that inducing singular equivalences or fully faithful functor between the singularity
categories.

If $G$ is a derived equivalence, then $G[i]$ is non-negative for some integer $i$. In general, the following lemma shows that adjoint pairs of triangle functors between derived categories may give rise to non-negative functors.

\begin{lem}\label{non}	
	Let
	\[\xymatrix{
		\D{\Modcat{B}}\ar@/^-0.8pc/[rr]^{G}&&\D{\Modcat{A}}\ar@/_0.8pc/[ll]_{F}&& }\]
	be an adjoint pair with both $F$ and $G$ preserving $\Kb{\sf proj}$. Then, up to shifts,  $G$ restricts to a non-negative functor  from $\Db{\Modcat{B}}$ to $\Db{\Modcat{A}}$.
\end{lem}
\begin{proof}
Since $F$ preserves $\Kb{\sf proj}$, it follows from  Lemma  \ref{re1} that $G$ preserves $\Db{\sf Mod}$ and coproducts.

By assumption both $G(B)$ and $F(A)$ are bounded complexes of finitely generated projectives. Assume that   $F(A)$ and $G(B)$  are of the following form:
\[G(B): \quad \cdots\rightarrow 0\rightarrow P^{-m}\rightarrow \cdots \rightarrow P^0\rightarrow P^1\rightarrow \cdots \rightarrow P^n\rightarrow 0\rightarrow \cdots, \]
\[F(A): \quad \cdots\rightarrow 0\rightarrow Q^{-r}\rightarrow \cdots \rightarrow Q^0\rightarrow Q^1\rightarrow \cdots \rightarrow Q^s\rightarrow 0\rightarrow \cdots.\]
Set $t={\rm max}\{m,s\}$, $\hat{F}=F[t]$ and $\hat{G}=G[-t]$. Then $\hat{G}(B)=G(B)[-t]$ is  a complex in $\Kb{\pmodcat{A}}$ with zero terms in all negative degrees, and $\hat{F}(A)=F[A][t]$ is  a complex in $\Kb{\pmodcat{B}}$ with zero terms in all positive degrees. Since $\hat{G}$ preserves coproducts, the complex $\hat{G}(P)$ is isomorphic to a complex in $\Kb{\Pmodcat{A}}$ with zero terms in all negative degrees for all $P\in \Pmodcat{B}$.

For any  $X\in \Modcat{B}$ and any integer $i$, we have
isomorphisms
\begin{align*}
	H^i(\hat{G}(X)) &\simeq \Hom_{\Db{\Modcat{A}}}(A,\hat{G}(X)[i])\\
	&\simeq \Hom_{\Db{\Modcat{B}}}(\hat{F}(A),X[i])\\
	&\simeq \Hom_{\Kb{\Modcat{B}}}(\hat{F}(A),X[i]),
	\end{align*}
where the second isomorphism follows from the adjointness of $\hat{F}$ and $\hat{G}$.
Therefore, $H^i(\hat{G}(X))=0$ for any $i<0$, that is, $\hat{G}(X)$ has no homology in negative degrees. This proves that $\hat{G}$ is non-negative.
\end{proof}

\begin{cor}\label{cd1}
Suppose that  $(F, G, H)$ in the following diagram is an adjoint triple of triangulated functors.
 \[\xymatrix{
\D{\Modcat{B}}\ar[rr]^{G}&&\D{\Modcat{A}}\ar@/^1pc/[ll]^{H}\ar@/_1pc/[ll]_{F}&& }\]
Assume that $G$ preserves $\Kb{\sf proj}$ and $H$ preserves $\Kb{\sf Proj}$.
Then we
have the commutative diagrams (\ref{d}).
Moreover, if $G$ induces a fully faithful functor between the singularity categories of $B$ and $A$, then ARC holds for $A$ implies that it holds for $B$.
\end{cor}
\begin{proof}
Since $G$ has a right adjoint $H$, by Lemma \ref{re1}, we have that
$F$ preserves $\Kb{\sf proj}$.
It follows from Lemma \ref{non} that, up to shifts,  $G$ restricts to a non-negative functor from $\Db{\Modcat{B}}$ to $\Db{\Modcat{A}}$. By Proposition \ref{cd}, one gets the commutative diagram (\ref{d}).
Moreover, $F$ preserves $\Kb{\sf proj}$ implies that $G$ preserves $\Db{\sf mod}$.
Taking $X\in \modcat{B}$ and using \cite[Section 4.2]{Hu2017aa},
we have that
$\overline{G}(X)$ is the cokernel of
some map appeared in
the projective resolution of $G(X)$.
Therefore, the fact that $G$ preserves $\Db{\sf mod}$
implies that $\overline{G}$ preserves finitely generated modules.
Now the statement follows from Lemma \ref{cdar}.
\end{proof}

A sequence of functors $(F_1, F_2,\cdots, F_r)$ between two categories is called an {\em adjoint tuple} if $(F_i, F_{i+1})$ is an adjoint pair for all $i=1, \cdots, r-1$.

\begin{thm}\label{ARE}
Let $A$ and $B$ be two algebras. Suppose that the sequence $(F_{1}, F_{2}, F_{3}, F_{4}, F_{5})$ of triangle functors in the following diagram is an adjoint tuple.
\[\xymatrix{
	\D{\Modcat{B}}\ar[rr]^{F_3}\ar@/_2pc/[rr]_{F_5}\ar@/^2pc/[rr]^{F_1}&&
\D{\Modcat{A}}\ar@/^1pc/[ll]^{F_4}\ar@/_1pc/[ll]_{F_2}&& }\]
Assume that $F_3$ induces a singular equivalence  between $B$ and $A$. If $F_4$ preserves $\Kb{\sf Proj}$, then ARC holds for $A$ if and only if it holds for $B$.
\end{thm}
\begin{proof}
By Lemma \ref{re1}, the functors $F_1, F_2$ and $F_3$ preserve $\Kb{\sf proj}$.
Now consider the adjoint triple $(F_2, F_3, F_4)$. Since $F_4$ preserves $\Kb{\sf Proj}$ and $F_3$ induces a singular equivalence, by Corollary \ref{cd1}, ARC holds for $A$ implies that it holds for $B$.

Conversely, consider the adjoint triple $(F_1, F_2, F_3)$. By Lemma \ref{re1},
the functors $F_2$ and $F_3$ preserve both $\Kb{\sf proj}$ and $\Db{\sf mod}$. Thus, $F_2$ and $F_3$ induce an adjoint pair between the singularity categories of $A$ and $B$. Since $F_3$ induces a singular equivalence, so does $F_2$. Moreover $F_3$ also preserves $\Kb{\sf Proj}$ since it preserves both $\Kb{\sf proj}$ and coproducts. Hence, applying Corollary \ref{cd1} to the adjoint triple $(F_1, F_2, F_3)$, we conclude that ARC holds for $B$ implies that it holds for $A$.
\end{proof}

\begin{cor}\label{6-func}
Let $A$ and $B$ be two algebras. Suppose that the sequence $(F_{1}, \cdots, F_{6})$ of triangle functors in the following diagram is an adjoint tuple.
\[\xymatrix{
	\D{\Modcat{B}}\ar[rr]^{F_3}\ar@/_2pc/[rr]_{F_5}\ar@/^2pc/[rr]^{F_1}&&\D{\Modcat{A}}\ar@/^1pc/[ll]^{F_4}\ar@/_1pc/[ll]_{F_2}\ar@/^3pc/[ll]^{F_6}&& }\]
Assume that $F_3$ induces a singular equivalence  between $B$ and $A$. Then ARC holds for $A$ if and only if it holds for $B$.
\end{cor}
\begin{proof}
	According to Theorem \ref{ARE}, it suffices to prove that $F_4$ preserves $\Kb{\sf Proj}$. However, this follows easily from the fact that $F_4$ preserves $\Kb{\sf proj}$ and coproducts by Lemma \ref{re1}.
\end{proof}

An immediate consequence of Corollary \ref{6-func} is that derived equivalences preserve ARC, which was proved by Wei in \cite{Wei2012aa}.  Actually, a derived equivalence $F$ and its quasi-inverse $F^{-1}$ give rise to an adjoint tuple $(F, F^{-1}, F, F^{-1}, \cdots )$ of arbitrary length, and $F$ induces a singular equivalence.

\medskip
Adjoint tuples typically occur in recollements and ladders.

\medskip
Let $\mathcal{T}_1$, $\mathcal{T}$ and $\mathcal{T}_2$ be
triangulated categories. A {\it recollement} \cite{Beuilinson1982aa} of $\mathcal{T}$
relative to $\mathcal{T}_1$ and $\mathcal{T}_2$ is a diagram
$$\xymatrix@!=4pc{ \mathcal{T}_1 \ar[r]|{i_*=i_!} & \mathcal{T} \ar @/_1pc/[l]|{i^*}
\ar @/^1pc/[l]|{i^!} \ar[r]|{j^!=j^*} & \mathcal{T}_2
\ar @/_1pc/[l]|{j_!} \ar @/^1pc/[l]|{j_*}}$$
of triangulated categories and triangle functors
such that

(1) $(i^*,i_*,i^{!}), (j_!,j^!, j_{*})$ are adjoint triples of triangle functors;

(2) $i_*$, $j_!$ and $j_*$ are full embeddings;

(3) $j^!i_*=0$ (and thus also $i^!j_*=0$ and $i^*j_!=0$);

(4) for each $X \in \mathcal {T}$, there are triangles

$$\begin{array}{l} j_!j^!X \rightarrow X  \rightarrow i_*i^*X  \rightarrow
\\ i_!i^!X \rightarrow X  \rightarrow j_*j^*X  \rightarrow
\end{array},$$ where the maps are given by adjunctions.

\medskip
A {\it ladder} \cite{Angeleri-Hugel2017aa} is a finite or infinite diagram of triangulated categories and triangle functors
$$
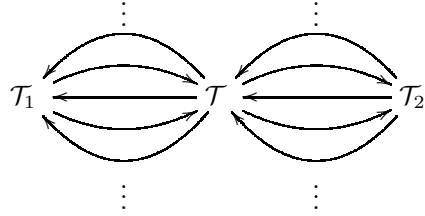
\xymatrix@!=4pc{ \mathcal{T}_1 \ \ar @/^1pc/[r] \ar @/_1pc/[r] & \mathcal{T}
\ar[l] \ar @/^2pc/[l]^\vdots \ar @/_2pc/[l]_\vdots
\ \ar @/^1pc/[r] \ar @/_1pc/[r]
 & \mathcal{T}_2 \ar[l] \ar @/^2pc/[l]^\vdots \ar @/_2pc/[l]_\vdots
}$$ such that any three consecutive rows form a recollement.
The {\it height} of a ladder is the number of recollements contained in it (counted with multiplicities).

\medskip
Now, we are ready to apply our discussion to ladders.

\begin{thm}\label{lad}
Suppose that $A$, $B$ and $C$ are algebras, and there is a ladder of height 3
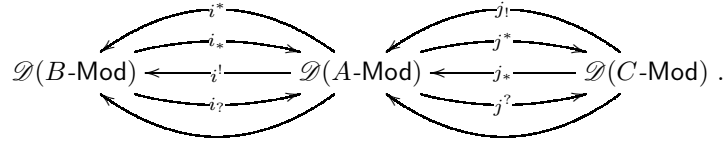
\[\xymatrix{
	\D{\Modcat{B}}\ar@/^1pc/[rr]|{i_*}\ar@/_1pc/[rr]|{i_?}&&
\D{\Modcat{A}} \ar[ll]|{i^!}\ar@/_2pc/[ll]|{i^*}\ar@/^2pc/[ll]
\ar@/^1pc/[rr]|{j^*}\ar@/_1pc/[rr]|{j^?}&&
\D{\Modcat{C}}\ar[ll]|{j_*}\ar@/_2pc/[ll]|{j_!}\ar@/^2pc/[ll]
}.\]

\medskip

\noindent
Then ARC holds for $A$ implies that it holds for $B$.
Moreover, if the ladder can be completed to a ladder of height $4$,
then the following statements hold.
\begin{enumerate}
	\item[{$(1)$}.] If the ARC holds for $A$, then it holds for $B$ and $C$;	
	\item[{$(2)$}.] $i_*$ (resp. $j^*$) induces a singular equivalence if and only if $C$ (resp. $B$) has finite global dimension,
and in this case, the ARC holds for $A$ if and only if it holds for $B$ (resp. $C$).
\end{enumerate}
\end{thm}
\begin{proof}
It follows from Lemma \ref{re1} that $i^*$, $i_*$ and $i^!$ preserve $\Kb{\sf proj}$,
and $i^{!}$ preserves coproducts. Therefore, $i^{!}$ preserves $\Kb{\sf Proj}$.
Moreover, $i_*$ is a fully faithful functor which preserve $\Kb{\sf proj}$ and
$\Db{\sf mod}$, and then the induced functor $\tilde{i}_* : \Dsg{B}\rightarrow \Dsg{A}$
is also fully faithful, see \cite[Lemma 1.2]{Orlov2004aa}.
Applying Corollary \ref{cd1} to the adjoint triple $(i^*,i_*,i^{!})$, we get that
the ARC holds for $B$ provided it holds for $A$.

Assume the ladder can be completed to a ladder of height 4. Without loss of generality,
we may assume it extend one step downwards, and the upward case can be proved similarly.
Consider the following ladder of height 4
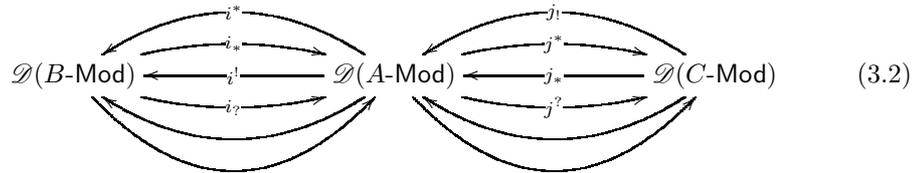
\begin{equation}\label{ladder}
\xymatrix@!=3pc{
	\D{\Modcat{B}}\ar@/^1pc/[rr]|{i_*}\ar@/_1pc/[rr]|{i_?}\ar@/_3pc/[rr]&&
\D{\Modcat{A}}\ar[ll]|{i^!}\ar@/_2pc/[ll]|{i^*}\ar@/^2pc/[ll]
\ar@/^1pc/[rr]|{j^*}\ar@/_1pc/[rr]|{j^?}\ar@/_3pc/[rr]&&
\D{\Modcat{C}}\ar[ll]|{j_*}\ar@/_2pc/[ll]|{j_!}\ar@/^2pc/[ll] }
\end{equation}

\vspace{1cm}
\noindent
Then, it follows from Lemma \ref{re1} that $j_*$ and $j^?$ preserve $\Kb{\sf proj}$.
Moreover, $j^?$ preserves coproducts and then it preserves $\Kb{\sf Proj}$.
On the other hand, $j_*$ induces a fully faithful functor between the
corresponding singularity categories.
Now, applying Corollary \ref{cd1} to the adjoint triple $(j^*,j_*,j^?)$,
we get that the ARC holds for $C$ if it holds for $A$.

By \cite[Proposition 2.5]{Liu2015aa},
the functors $i_*,i^!,i_?,j^*,j_*$ and $j^?$ induce a recollement between
the corresponding singularity categories.
Therefore, $i_*$ induces a singular equivalence if and only if
$\Dsg{C}=0$, and this occurs precisely when
$C$ has finite global dimension.
In this case, $i^!$ also induces a singular equivalence.
Applying Corollary~\ref{6-func} to the left
part of the ladder~\ref{ladder}, we obtain that the ARC holds for $A$ if and only if it holds for $B$.
Similarly, $j^*$ induces a singular equivalence if and only if $B$ has finite global dimension,
and in this case, the ARC holds for $A$ if and only if it holds for $C$.
\end{proof}

Theorem \ref{lad} can be applied to triangle matrix algebras.
\begin{cor}\label{tri-alg}
Let $A=\left(
\begin{array}{cccc}
B& 0  \\
_CM_B& C \\
\end{array}
\right)$ be a triangular matrix algebra, where $B, C$ are algebras and $M$ a finitely generated $C$-$B$-bimodules.
Then following statements hold.
\begin{enumerate}
	\item[$(1)$.] If $\gldim B<\infty$ and $\projdim_CM<\infty$, then the ARC holds for $A$ if and only if it holds for $C$;
	\item[$(2)$.] If $\gldim C<\infty$ and $\projdim (M)_B<\infty$, then the ARC holds for $A$ if and only if it holds for $B$.
\end{enumerate}
\end{cor}
\begin{proof}
Let $e_1=\left(
\begin{array}{cccc}
1 & 0  \\
0& 0 \\
\end{array}
\right)$ and $e_2=\left(
\begin{array}{cccc}
0 & 0  \\
0& 1 \\
\end{array}
\right)$. It follows from \cite[Example 3.4]{Angeleri-Hugel2017aa} that there is a
ladder of height $2$
\begin{equation}\label{la2}
\xymatrix{
	\D{\Modcat{C}}\ar@/^0.7pc/[rr]^{Ae_2\otimesL_C-}\ar@/_2pc/[rr]&&
\D{\Modcat{A}}\ar@/^0.7pc/[rr]
\ar@/^0.7pc/[ll]_{e_2A\otimesL_A-}\ar@/_2pc/[ll]\ar@/_2pc/[rr]&&
\D{\Modcat{B}}\ar@/^0.7pc/[ll]\ar@/_2pc/[ll]_{Ae_1\otimesL_B-}.}
\end{equation}

\vspace{0.7cm}
\noindent
If $\gldim B<\infty$ and $\projdim_CM<\infty$, then $\projdim M_B<\infty$
and by \cite[Example 3.4]{Angeleri-Hugel2017aa}, the ladder~(\ref{la2}) can be extended to a ladder of height $4$.
Applying Theorem~\ref{lad} to this ladder, we have that
the ARC holds for $A$ if and only if it holds for $C$.
The second case can be proved similarly.
\end{proof}

Recall that an $A$-module $M$ is {\em Gorenstein projective} if there are short exact sequences
$$0\lra X_i\lra P_i\lra X_{i+1}\lra 0, i\in\mathbb{Z}$$
in $\modcat{A}$ with $X_i\in {}^{\perp}A$ for all $i$ such that $M=X_0$. Particularly, all Gorenstein projective modules are contained in ${}^{\perp}A$. Denote by $\GPmodcat{A}$ the full subcategory of $\modcat{A}$ consisting
of all Gorenstein projective modules. The stable category of $\GPmodcat{A}$ is a triangulated category and is contained in $\underline{{}^{\perp}A}$. Thus, one gets full embeddings
$$\underline{\GPmodcat{A}}\hookrightarrow \underline{{}^{\perp}A}\hookrightarrow\Dsg{A}.$$
If ARC holds for $A$, that is, $\underline{{}^{\perp}A}$ does not contain any nonzero self-orthogonal objects, then there is no nonzero self-orthogonal objects in $\underline{\GPmodcat{A}}$.  It is then natural to conjecture:

\medskip
{\bf Gorenstein projective Conjecture} ({\bf GPC} for short):  A finitely generated Gorenstein projective
module $M$ over $A$ is projective if
 $\Ext ^i _A(M,M)
= 0$, for any $i \geq 1$.

\medskip
This conjecture was stated by Luo and Huang in \cite{Luo2008aa}.
Let $G:\Db{\modcat{B}}\ra\Db{\modcat{A}}$
be a non-negative triangle functor (a small module version of
definition~\ref{def-non}) admitting a right adjoint $H$ which preserving $\Kb{\rm proj}$.
Due to \cite[Proposition 5.3]{Hu2017aa}, we have
a  commutative diagram
$$ \xymatrix@C=1.5cm{
		\underline{\GPmodcat{B}} \ar[r]^{\overline{G}} \ar @{^{(}->}[d]^{\iota_B}
& \underline{\GPmodcat{A}} \ar @{^{(}->}[d]^{\iota_A}  \\
		 \Dsg{B}  \ar[r]^{\tilde{G}}   &\Dsg{A},}$$
\noindent analogous to Proposition~\ref{cd}.
Therefore, if we replace ARC by GPC, all results
in this section (e.g., Theorem \ref{ARE} and Theorem~\ref{lad}) still hold and the proofs are almost identical to the case of ARC.

\section{Singular equivalences induced by tensor functors}
Let $A$ and $B$ be two algebras, and $\cpx{X}$ be a complex of $B$-$A$-bimodules.
Then the tensor product functor $F=\cpx{X}\otimesL_A-:\D{\Modcat{A}}\rightarrow\D{\Modcat{B}}$
has a right adjoint $\RHom_B(\cpx{X},-)$. It is natural to ask when
$F$ induces an equivalence between the singularity categories,
and under which conditions this singular equivalence preserves ARC. 

It does happen often that $F$ induces a singular equivalence. Typical examples are singular equivalences of Morita type (with level) when $\cpx{X}$ is a module (\cite[Theorem 3.1]{Zhou2013aa},\cite[Definition 2.1]{Wang2015aa}). See \cite[Proposition 4.8]{Chen2020} for recent progress in this direction. In the general case, it was proved recently by  Dalezios \cite[Theorem 3.6]{Dalezios2020} that $F$ induces a singular equivalence if $\cpx{Y}:= \RHom_B(\cpx{X},B)$ is a perfect complex of $A$-modules, $\RHom_B(\cpx{X},\cpx{X})
\cong A$ in $\Dsg{A^{e}}$, and $\cpx{X}\otimesL_A\RHom_B(\cpx{X},B)\cong B$ in $\Dsg{B^{e}}$.

\medskip
The main result of this section is the following theorem.

\begin{thm}\label{atar}
Suppose that $\cpx{X}$ is a $B$-$A$-bimodules complex which is perfect over $B$ and $A$,
and assume $\cpx{Y}:= \RHom_B(\cpx{X},B)$ is a perfect complex of $A$-modules.
If ${}_B\cpx{X}\otimes _A^L-:\D{\Modcat{A}}\rightarrow\D{\Modcat{B}}$
induces a singular equivalence, then ARC holds for $B$ implies that it holds for $A$. If moreover
$\RHom_A(\cpx{Y}, A)\in \Kb{\pmodcat{B}}$, then ARC holds for $A$ if and only if it holds for $B$.
\end{thm}

\begin{proof}
Note that there is an adjoint tuple
\[\xymatrix{
	\D{\Modcat{B}}
	\ar@/_0.7pc/[rr]^{
    \cpx{Y}\otimesL_B-	
}\ar@/^2pc/[rr]^{L}&& \D{\Modcat{A}}\ar@/^2pc/[ll]^{\RHom_A(\cpx{Y}, -)}\ar@/_0.7pc/[ll]_{\cpx{X}\otimesL_A-}&& }\]
with $\cpx{X}\otimesL_A-$ and $\cpx{Y}\otimesL_B-$ preserving $\Kb{\sf proj}$. The existence of $L$ follows from the assumption that $\cpx{X}_A$ is a perfect complex. Clearly, $\cpx{Y}\otimesL_B-$ preserves coproducts and thus $\cpx{Y}\otimesL_B-$
preserves $\Kb{\sf Proj}$. Now notice that
$\cpx{X}\otimesL_A-$ induce a singular equivalence between $B$ and $A$,
by Corollary~\ref{cd1}, ARC holds for $B$ implies that it holds for $A$.

Since $\cpx{Y}\otimesL_B-$ preserves $\Kb{\sf proj}$, it follows from
Lemma \ref{re1} that
$\RHom_A(\cpx{Y}, -)$ preserves coproducts,
and it admits a right adjoint $K$.
Hence, the condition $\RHom_A(\cpx{Y}, A)$
\linebreak
$\in \Kb{\pmodcat{B}}$
implies that $\RHom_A(\cpx{Y}, -)$
preserves $\Kb{\sf Proj}$. On the other hand,
$(\cpx{X}\otimesL_A-, \cpx{Y}\otimesL_B-)$ gives an adjoint
pair between the singularity categories.
Since $\cpx{X}\otimesL_A-$ induces a singular equivalence,
so does $\cpx{Y}\otimesL_B-$.
Therefore, there is an adjoint tuple
\[\xymatrix{
	\D{\Modcat{B}}\ar[rr]^{\cpx{Y}\otimesL_B-}\ar@/_2.5pc/[rr]_{K}\ar@/^2.5pc/[rr]^{L}&&\D{\Modcat{A}}\ar@/^1pc/[ll]^{\RHom_A(\cpx{Y}, -)}\ar@/_1pc/[ll]_{\cpx{X}\otimesL_A-}&& }\]
with $\RHom_A(\cpx{Y}, -)$ preserving $\Kb{\sf Proj}$ and $\cpx{Y}\otimesL_B-$ inducing a singular equivalence between $B$ and $A$. Then we have done by Theorem \ref{ARE}.
\end{proof}

Now we focus on singular equivalences from change of rings, which was studied in \cite{Oppermann2019aa, Dalezios2020}. Let $f : A\rightarrow B$ be a morphism of algebras with $\projdim_AB<\infty$ and $\projdim B_A<\infty$. Recall from \cite[Lemma 3.2]{Oppermann2019aa} that there is an adjoint pair
\begin{align}\label{ad1}
\xymatrix{
	\D{\Modcat{B}}\ar@/^-0.7pc/[rr]_{{}_AB\otimesL_B-}&&\D{\Modcat{A}}.\ar@/_0.7pc/[ll]_{{}_BB\otimesL_A-}&& }
\end{align}
If ${\rm Cone}(f)\in \Kb{\pmodcat{A^e}}$ (that is, $B\cong A$ in $\Dsg{A^{e}}$)
and ${}_BB\otimesL_AB_B\overset{\sim}{\rightarrow}{}_BB_B$ in $\Dsg{B^{e}}$,
then ${}_BB\otimesL_A-$ and ${}_AB\otimesL_B-$ induce
mutual equivalences between the singularity categories; see \cite[Corollary 3.7]{Dalezios2020}.

In particular, $f : A\rightarrow B$ is a \emph{homological epimorphism} if the induced functor
\[f_*={}_AB\otimesL_B-: \D{\Modcat{B}}\rightarrow \D{\Modcat{A}}\]
is a full embedding, or equivalently, there is an isomorphism ${}_BB\otimesL_AB_B\overset{\sim}{\rightarrow}{}_BB_B$ in $\D{\Modcat{B^e}}$; see \cite{Geigle1991aa}. In this case, ${}_AB\otimesL_B-$ induces a fully faithful functor between the corresponding singularity
categories.

\begin{cor}\label{cor-ring-hom}
Let $f : A\rightarrow B$ be a morphism of algebras with $\projdim_AB<\infty$ and $\projdim B_A<\infty$.
Then the following hold.
\begin{enumerate}
	\item[{$(1)$}.] If ${\rm Cone}(f)\in \Kb{\pmodcat{A^e}}$, then ARC holds for $B$ implies that it holds for $A$;
    \item[{$(2)$}.] If moreover $f$ is a homological epimorphism and $\RHom_A(B, A)\in \Kb{\pmodcat{B}}$,
then the ARC holds for $A$ if and only if it holds for $B$.
\end{enumerate}
\end{cor}

\begin{proof}
(1) Since $\projdim B_A<\infty$, it follows from \cite[Lemma 2.8]{Angeleri-Hugel2017aa}
that the functor ${}_BB\otimesL_A-$ in  diagram {\rm (\ref{ad1})}
has a left adjoint. Moreover, $\projdim_AB<\infty$ implies that ${}_AB\otimesL_B-$
preserves $\Kb{\sf Proj}$, and ${\rm Cone}(f)\in \Kb{\pmodcat{A^e}}$ yields that
${}_BB\otimesL_A-$ induces a fully faithful functor between the corresponding singularity
categories, see \cite[Proposition 3.7]{Oppermann2019aa}.
Now the statement (1) follows immediately from Corollary~\ref{cd1}.

(2) Since ${\rm Cone}(f)\in \Kb{\pmodcat{A^e}}$ and $f$
is a homological epimorphism, we infer that
 ${}_BB\otimesL_A-$ and ${}_AB\otimesL_B-$ induce mutual equivalences between the singularity categories.
Now we finish our proof by Theorem~\ref{atar}.
\end{proof}

 A special case of algebra homomorphism is the canonical map from an algebra $A$ to its quotient $A/I$ for some ideal $I$.
Applying Corollary~\ref{cor-ring-hom}, we get the following results:
\begin{enumerate}
	\item[{$(1)$}.] If $I$ has finite projective dimension as an $A$-$A$-bimodule, then
ARC holds for $A/I$ implies that it holds for $A$. Indeed,
$\projdim_{A^e}I<\infty$ yields that $\projdim _A(A/I)<\infty$ and $\projdim (A/I)_A<\infty$.
Moreover, it is clear that the cone of
$A\rightarrow A/I$ is $I[1]$, which belongs to $\Kb{\pmodcat{A^e}}$ by assumption. So, the statement follows from
Corollary~\ref{cor-ring-hom} (1).

    \item[{$(2)$}.] A special case of (1) occurs when $I\cong M\otimes _k N$, where
$M$ and $N$ are left and right $A$-modules, respectively, both have finite projective dimension.
Another example is that $I$ is a $1$-dimensional ideal
which has finite projective dimension as left and as right module.
We refer to \cite[Corollary 3.10 and 3.11]{Oppermann2019aa} for more
explanations.

    \item[{$(3)$}.] In general, it is difficult to check
whether the condition $\RHom_A(B, A)\in \Kb{\pmodcat{B}}$
in Corollary~\ref{cor-ring-hom} (2) is satisfied. However,
it is the case when $A$ is a Gorenstein algebra.
In fact, assume all conditions in Corollary~\ref{cor-ring-hom} are satisfied,
except for $\RHom_A(B, A)\in \Kb{\pmodcat{B}}$.
Then it follows from the proof of Theorem~\ref{atar} that there are
adjoint functors
\[\xymatrix{
	\D{\Modcat{B}}\ar[rr]|{i_*=B\otimesL_B-}\ar@/_2pc/[rr]\ar@/^2pc/[rr]&&\D{\Modcat{A}}\ar@/^1pc/[ll]|{i^!}
\ar@/_1pc/[ll]|{i^*=B\otimesL_A-}&& }\]
where the functor $i_*$ is fully faithful.
Therefore, we get $i^*i_*\cong 1_B$ and $i^!i_*\cong 1_B$,
and thus $B\cong i^!i_*B\in i^!(\Kb{\pmodcat{A}})=i^!(\Kb{\imodcat{A}})$, which is contained in $\Kb{\imodcat{B}}$
by \cite[Lemma 1]{Qin16}. As a result, we obtain $\injdim _BB<\infty$, and dually,
we also have $\projdim _B\Hom _k(B,k)<\infty$, that is, $B$ is also a Gorenstein algebra.
Therefore, $\RHom
_A(B,A)=i^!A\in i^!(\Kb{\imodcat{A}})\subseteq \Kb{\imodcat{B}}=\Kb{\pmodcat{B}}$.

\end{enumerate}

%

\section{Examples}

In this section, we illustrate our results by some examples.

\begin{exm} (Tiled orders) Let $A$ be a finite dimensional algebra, $I_{i, j}$ be an ideal of $A$. Let us consider tiled triangular rings, i.e., rings of the form\\
$$\Delta=\small
\begin{pmatrix}
A&I_{1,2}&\cdots&I_{1,n}\\
A&A&&\vdots\\
\vdots&&\ddots&I_{n-1,n}\\
A&\cdots&\cdots&A
 \end{pmatrix}.$$

 \begin{prop}

Assume that $\projdim({}_AI_{1,i})<\infty, \gldim(A/I_{i-1, i})<\infty$ for $i=2, \cdots, n$, and $A$ satisfies ARC. Then $\Delta$ satisfies ARC.
  \end{prop}

 \begin{proof}

 As shown in \cite[Proposition 4.14]{Chen2012ah},
the two algebras
$$\Delta=\small\begin{pmatrix}A&I_{1,2}&I_{1,3}&\cdots&I_{1,n}\\
A&A&I_{2,3}&&\vdots\\
\vdots&\vdots&\ddots&\ddots&\vdots\\
\vdots&\vdots&&\ddots&I_{n-1,n}\\
A&A&A&A&A
\end{pmatrix}\quad\mbox{ and }\quad
\Phi=\small\begin{pmatrix}A/I_{1,2}&I_{2,3}/I_{1,3}&\cdots&I_{2,n}/I_{1,n}&0\\
A/I_{1,2}&A/I_{1,3}&&&\vdots\\
\vdots&\vdots&&\ddots&\vdots\\
A/I_{1,2}&A/I_{1,3}&\cdots&A/I_{1,n}&0\\
A/I_{1,2}&A/I_{1,3}&\cdots&A/I_{1,n}&A
\end{pmatrix}$$ are derived equivalent.
Let us denote $\Phi_1:=$ $$\begin{pmatrix}A/I_{1,2}&I_{2,3}/I_{1,3}&\cdots&I_{2,n}/I_{1,n}\\
A/I_{1,2}&A/I_{1,3}&\cdots&I_{3,n}/I_{1,n}\\
\vdots&\vdots&\ddots&\vdots\\
A/I_{1,2}&A/I_{1,3}&\cdots&A/I_{1,n}\end{pmatrix}.$$

Since $\projdim({}_AI_{1,i})<\infty$ for $i=2, \cdots, n$, $(A/I_{1,2}, A/I_{1,3}, \cdots, A/I_{1,n})$ has finite projective dimension as a left $A$-module.

 Now, we want to show that $\Phi_1$ has finite global dimension.
Let $e$ be an idempotent of $\Phi_1$ which has $1$ in the $(1, 1)$-th position and zeros elsewhere. It is immediate that
$$\Phi_1e\Phi_1=\begin{pmatrix}A/I_{1,2}&I_{2,3}/I_{1,3}&\cdots&I_{2, n}/I_{1,n}\\
A/I_{1,2}&I_{2,3}/I_{1,3}&\cdots&I_{2, n}/I_{1,n}\\
\vdots&\vdots&&\vdots\\
A/I_{1,2}&I_{2,3}/I_{1,3}&\cdots&I_{2, n}/I_{1,n}

\end{pmatrix}
$$
which is projective as a right $\Phi_1$-module. Note that $\Delta$ is a ring. Then, by the multiplication of $\Delta$, we have $I_{1,2}I_{2, i}\subseteq I_{1,i}$ for $i=3, \cdots,n$. So, $I_{2, i}/I_{1,i}$ is a left $A/I_{1,2}$-module for $i=3,\cdots, n$. It follows from $\gldim(A/I_{1,2})<\infty$ that $I_{2,i}/I_{1,i}$ has finite projective dimension as a left $A/I_{1,2}$-module for $i=3, \cdots, n$. Hence, $\Phi_1e\Phi_1$ has finite projective dimension as a left $\Phi_1$-module.

It follows from $\Tor^{\Phi_1}_i(\Phi_1/\Phi_1e\Phi_1, \Phi_1/\Phi_1e\Phi_1)=0$ for $i>0$ that the canonical ring homomorphism $\lambda:{}\Phi_1\ra \Phi_1/\Phi_1e\Phi_1$ is a homological epimorphism. Thus, we have an adjoint triple
\[\xymatrix{
	\D{\rm{Mod}\text{-}{\Phi_1/\Phi_1e\Phi_1}}\ar[rr]^{\lambda_*}&&\D{\rm{Mod}\text{-}{\Phi_1}}\ar@/^1pc/[ll]^{\RHom_{\Phi_1}(\Phi_1/\Phi_1e\Phi_1, -)}\ar@/_1pc/[ll]_{
	-\otimesL_{\Phi_1}\Phi_1/\Phi_1e\Phi_1}&& }\]
	where $\lambda_*$ is an embedding, and $-\otimesL_{\Phi_1}\Phi_1/\Phi_1e\Phi_1$ and $\RHom_{\Phi_1}(\Phi_1/\Phi_1e\Phi_1, -)$ are the derived functors of $-\otimes_{\Phi_1}\Phi_1/\Phi_1e\Phi_1$ and $\Hom_{\Phi_1}(\Phi_1/\Phi_1e\Phi_1, -)$, respectively. Note that $\Phi_1e\Phi_1$ is projective as a right $\Phi_1$-module, and has finite projective dimension as a left $\Phi_1$-module. Then, the adjoint triple $(-\otimesL_{\Phi_1}\Phi_1/\Phi_1e\Phi_1, \lambda_*, \RHom_{\Phi_1}(\Phi_1/\Phi_1e\Phi_1, -))$
restricts to $\Db{\rm{mod}}$ and $\Kb{\rm{proj}}$, respectively.

By \cite[Example 5.3.4]{Nicolas2007}, the Verdier localization of $\D{\rm{Mod}\text{-}\Phi_1}$ via the essential image of $\D{\rm{Mod}\text{-}\Phi_1/\Phi_1e\Phi_1}$ under $\lambda_*$ is triangle equivalent to $\rm{Tria}_{\D{\rm{Mod}\text{-}\Phi_1}}(\Phi_1e\Phi_1)$ which is the smallest full triangulated subcategory of $\D{\rm{Mod}\text{-}\Phi_1}$ containing $\Phi_1e\Phi_1$ and closed under small coproducts, and $\rm{Tria_{\D{\rm{Mod}\text{-}\Phi_1}}}(\Phi_1e\Phi_1)$ is triangle equivalent to the category $\D{(\mathcal{C}_{dg}\Phi_1)(\Phi_1e\Phi_1, \Phi_1e\Phi_1)}$ in which $(\mathcal{C}_{dg}\Phi_1)(\Phi_1e\Phi_1, \Phi_1e\Phi_1)$ is a dg algebra.
Since $\Phi_1e\Phi_1$ is a finitely generated projective right $\Phi_1$-module, we have a triangle equivalence between $\D{(\mathcal{C}_{dg}\Phi_1)(\Phi_1e\Phi_1, \Phi_1e\Phi_1)}$ and $\D{\rm{Mod}\text{-}H^0((\mathcal{C}_{dg}\Phi_1)(\Phi_1e\Phi_1, \Phi_1e\Phi_1))}$. And the latter one is triangle equivalent to $\D{\rm{Mod}\text{-}e\Phi_1e}$. Note that $e\Phi_1e=A/I_{1,2}$ and $\gldim(A/I_{1,2})<\infty$. Then, $\Phi_1^{\opp}$ and $\Phi_1^{\opp}/\Phi_1e\Phi_1^{\opp}$ are singularly equivalent, where $\Phi_1^{\opp}$ and $\Phi_1^{\opp}/\Phi_1e\Phi_1^{\opp}$ are the opposite algebras of $\Phi_1$ and $\Phi_1/\Phi_1e\Phi_1$, respectively. Then, $\Phi_1$ has finite global dimension if and only if so does $\Phi_1/\Phi_1e\Phi_1$.

It is clear that
$$\Phi_1/\Phi_1e\Phi_1=\begin{pmatrix}
A/I_{2,3}&I_{3,4}/I_{2,4}&\cdots&I_{3,n}/I_{2,n}\\
A/I_{2,3}&A/I_{2,4}&\cdots&I_{4,n}/I_{2,n}\\
\vdots&\vdots&\ddots&\vdots\\
A/I_{2,3}&A/I_{2,4}&\cdots&A/I_{2,n}
\end{pmatrix}.
$$

We write $\Phi_2$ for $\Phi_1/\Phi_1e\Phi_1$. Recursively, $\Phi_1$ has finite global dimension if and only if so does $A/I_{n-1,n}$. Thus, we have $\gldim(\Phi_1)<\infty$. Hence, by Corollary \ref{tri-alg}, $\Delta$ satisfies ARC if and only if it holds for $A$ since derived equivalences preserve ARC.
\end{proof}

\end{exm}

\begin{exm}\label{one-point extensions}(One-point extensions)
	Let $A$ be an algebra, and let $M$ be a left $A$-module. The one-point extension algebra $A[M]$ is defined to be the triangular matrix algebra
	$$\begin{bmatrix}
	  k &0\\
	  M & A
	\end{bmatrix}$$
If $M$ has finite projective dimension, then ARC holds for $A[M]$ if and only if it holds for $A$. This follows immediately from Corollary \ref{tri-alg}.
\end{exm}

\begin{exm}\label{hereditary ideals}(Quotient algebras)
	Let $A$ be an algebra, and let $e$ be a primitive idempotent in $A$ such that the multiplication map $Ae\otimes_keA\ra AeA$ is an isomorphism and $eAe\simeq k$. The ideal $AeA$ is called a heredity ideal in the literature. If furthermore the injective dimension of the right $A$-module $eA$ is finite, then ARC holds for $A$ if and only if it holds for $A/AeA$.
	
Indeed, it follows from \cite{Cline1996} that $\D{\Modcat{A}}$ admits a recollement
\begin{align}\label{idem}
\xymatrix@!=6pc{ \D{\Modcat{A/AeA}} \ar[r]|{i_*} & \D{\Modcat{A}} \ar @/_1pc/[l]|{i^*}
\ar @/^1pc/[l]|{i^!} \ar[r]|{j^*} & \D{\Modcat{eAe}}
\ar @/_1pc/[l]|{j_!} \ar @/^1pc/[l]|{j_*}}
\end{align}
where $i^*=A/AeA\otimes_A^L-$, $i_*=A/AeA\otimes_{A/AeA}^L-$, $i^!
=\RHom_{A}(A/AeA, -)$,
$j_!=Ae\otimes_{eAe}^L-$, $j^*=eA\otimes_A^L-$ and $j_*=\RHom_{eAe}(eA, -)$.
Clearly, $\gldim eAe<\infty$ implies that $j^*A=eA\in\Kb{\pmodcat{eAe}}$,
and $\injdim eA_A<\infty$ yields that $j_*(eAe)=\RHom_{eAe}(eA, eAe)
=\RHom_{k}(eA,k)\in\Kb{\pmodcat{A}}$. Hence,
by \cite[Lemma 2.5 and Proposition 3.2]{Angeleri-Hugel2017aa},
the recollement {\rm (\ref{idem})} can be extended two steps downwards.
Similarly, by
\cite[Lemma 2.8 and Proposition 3.2]{Angeleri-Hugel2017aa},
$\projdim Ae_{eAe}<\infty$ implies that {\rm (\ref{idem})} can be extended one step upwards.
Therefore, {\rm (\ref{idem})} is completed to a ladder of height $4$,
and by Theorem~\ref{lad}, ARC holds for $A$ if and only if it holds for $A/AeA$.

\end{exm}

\begin{exm}(Derived discrete algebras) From \cite{Voss01}, an algebra
$A$ is said to be {\it derived discrete} provided for
every positive element $\mathbf{d}  \in
K_0(A)^{(\mathbb{Z})}$ there are only finitely many isomorphism
classes of indecomposable objects $X$ in $\mathcal{D}^b(A)$ of
cohomology dimension vector $(\underline{\dim} H^p(X))_{p \in \mathbb{Z}}
= \mathbf{d}$. Now we apply our main theorem
to show that ARC holds for all derived discrete algebras.

From \cite{BGS04,Voss01}, a basic connected derived
discrete algebra $A$ is derived equivalent to either a piecewise hereditary
algebra of Dynkin type, or a bound quiver algebra $\Lambda(r,n,m)$ given by
$$
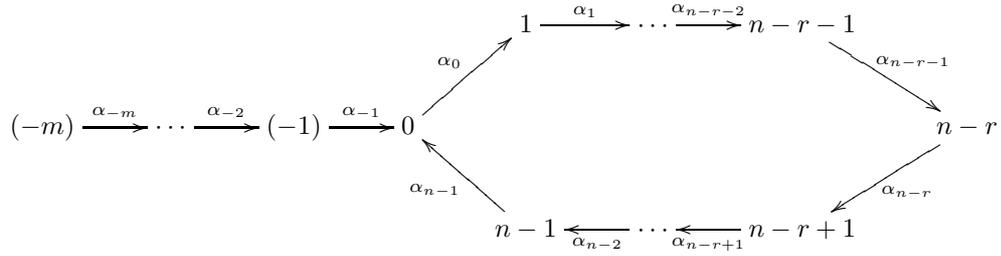
\xymatrix{ &&&& 1 \ar[r]^-{\alpha_{1}}
& \cdots \ar[r]^-{\alpha_{n-r-2}} &  n-r-1 \ar[dr]^-{\alpha_{n-r-1}}
& \\ (-m) \ar[r]^-{\alpha_{-m}} & \cdots \ar[r]^-{\alpha_{-2}} &
(-1) \ar[r]^-{\alpha_{-1}} & 0 \ar[ur]^-{\alpha_0} &&&& n-r
\ar[dl]^-{\alpha_{n-r}} \\ &&&& n-1 \ar[ul]^-{\alpha_{n-1}} & \cdots
\ar[l]^-{\alpha_{n-2}} & n-r+1 \ar[l]^-{\alpha_{n-r+1}} & }$$ with
the relations $\alpha_{n-1}\alpha_0, \alpha_{n-2}\alpha_{n-1},
\cdots , \alpha_{n-r}\alpha_{n-r+1}$, where $1 \leq r \leq
n$ and $m \geq 0$. Clearly, if $\gldim A<\infty$ then ARC always hold true,
and if $\gldim A=\infty$ then $A$ is derived equivalent to $\Lambda(n,n,m)$,
which admits a series of infinite ladders,
see \cite[Lemma 17]{Q16}. Note that the right terms of these ladders
are $k$, and the left can be reduced to $\Lambda(n,n,0)$ consecutively.
Hence, applying Theorem~\ref{lad} shows that ARC holds for $A$ if
and only if it holds for $\Lambda(n,n,0)$, and the latter is known as
$2$-truncated cycle algebra, which is
representation finite and then satisfies ARC by \cite{Auslander1975ae}.

\end{exm}

\vskip 10pt

\noindent {\bf Acknowledgements.}\quad  The paper is started when the authors are visiting University of Stuttgart; they thank Professor Steffen Koenig for his hospitality. This work is supported by National Natural Science
Foundation of China (No.s 11701321, 12031014, 11301398, 12061060 and 11901551).



\vskip 10pt
{\footnotesize
\noindent Yiping Chen

\medskip
\noindent
	School of Mathematics and Statistics, Wuhan University, Wuhan, 430072, China

	\medskip
\noindent
Hubei Key Laboratory of Computational Science (Wuhan University), Wuhan, Hubei, 430072, P. R. China

\noindent
Email: ypchen@whu.edu.cn
}

\bigskip
{\footnotesize \noindent Wei Hu
	
\medskip
\noindent	
School of Mathematical Sciences, Laboratory of Mathematics and Complex Systems, MOE, Beijing Normal University, 100875 Beijing, China

\noindent
Email: huwei@bnu.edu.cn
}

\bigskip

\footnotesize\noindent
Yongyun Qin

\medskip
\noindent
 College of Mathematics and Statistics,
Qujing Normal University, \\ Qujing, Yunnan 655011, China.

\medskip
\noindent
E-mail:
qinyongyun2006@126.com

\bigskip
 {\footnotesize \noindent Ren Wang
 	
 \medskip
 \noindent
 School of Mathematical Sciences, Key Laboratory of Wu Wen-Tsun Mathematics, University of Science and Technology of China, Hefei 230026, China

\noindent
Email: renw@mail.ustc.edu.cn
}

\end{document}